\documentclass[12pt]{amsart}
\usepackage{verbatim}
\usepackage{amssymb}
\usepackage{upref}
\usepackage[all]{xy}
\usepackage[usenames,dvipsnames]{color}

\allowdisplaybreaks

\newtheorem{thm}{Theorem}[section]
\newtheorem*{thm*}{Theorem}
\newtheorem{lem}[thm]{Lemma}
\newtheorem{cor}[thm]{Corollary}
\newtheorem{prop}[thm]{Proposition}

\theoremstyle{definition}

\newtheorem{notn}[thm]{Notation}

\newtheorem*{notn*}{Notation}

\newtheorem*{hyp*}{Hypothesis}

\newtheorem{rem}[thm]{Remark}
\newtheorem*{rem*}{Remark}

\numberwithin{equation}{section}

\newcommand{\thmref}[1]{Theorem~\textup{\ref{#1}}}
\newcommand{\corref}[1]{Corollary~\textup{\ref{#1}}}
\newcommand{\lemref}[1]{Lemma~\textup{\ref{#1}}}

\newcommand{\midtext}[1]{\quad\text{#1}\quad}
\newcommand{\righttext}[1]{\quad\text{#1 }}

\newcommand{\T}{\mathbb T}

\newcommand{\CC}{\mathcal C}

\newcommand{\KK}{\mathcal K}

\newcommand{\LL}{\mathcal L}

\newcommand{\OO}{\mathcal O}

\newcommand{\TT}{\mathcal T}

\renewcommand{\AA}{\mathcal A}

\renewcommand{\epsilon}{\varepsilon}

\DeclareMathOperator*{\spn}{span}
\DeclareMathOperator*{\clspn}{\overline{\spn}}

\newcommand{\case}{& \text{if }}

\newcommand{\<}{\langle}
\renewcommand{\>}{\rangle}

\newcommand{\inv}{^{-1}}

\renewcommand{\bar}{\overline}

\newcommand{\wilde}{\widetilde}

\newcommand{\ann}{^\perp}

\newcommand{\lpg}{\ell^p(G)}
\newcommand{\Lpg}{L^p(G)}

\newcommand{\czerog}{c_0(G)}

\newcommand{\inn}[2]{\langle #1,#2 \rangle}

\begin{document}
\title[A new look at crossed product correspondences]{A new look at crossed product correspondences and associated C*-algebras}

\author[B\'edos]{Erik B\'edos}
\address{Institute of Mathematics
\\University of Oslo
\\PB 1053 Blindern, 0316 Oslo, Norway}
\email{bedos@math.uio.no}

\author[Kaliszewski]{S. Kaliszewski}
\address{School of Mathematical and Statistical Sciences
\\Arizona State University
\\Tempe, Arizona 85287}
\email{kaliszewski@asu.edu}

\author[Quigg]{John Quigg}
\address{School of Mathematical and Statistical Sciences
\\Arizona State University
\\Tempe, Arizona 85287}
\email{quigg@asu.edu}

\author[Robertson]{David Robertson}
\address{School of Mathematics and Applied Statistics
\\University of Wollongong
\\NSW 2522, Australia}
\email{dave84robertson@gmail.com}



\date{\today}

\begin{abstract}

When a locally compact  group acts on a $C^*$-correspondence, it also acts on the associated Cuntz-Pimsner algebra in a natural way. Hao and Ng have shown that when the group is amenable the Cuntz-Pimsner algebra of the crossed product correspondence is isomorphic to the crossed product of the Cuntz-Pimsner algebra. In this paper, we have a closer look at this isomorphism in the case where the group is not necessarily amenable. We also consider what happens at the level of Toeplitz algebras.
\end{abstract}
\maketitle

\section{Introduction}
Suppose that a locally compact group $G$ acts on a nondegenerate $C^*$-correspondence $(X,A)$. By universality of the Cuntz-Pimsner algebra $\OO_X$, $G$ also acts on $\OO_X$ and 
the \emph{Hao-Ng isomorphism} \cite[Theorem~2.10]{HN} says that if $G$ is amenable, then
\begin{equation}\label{iso}
\OO_X\rtimes G\cong \OO_{X\rtimes G}\,.
\end{equation}
We'd like to know whether the Hao-Ng isomorphism holds without the amenability hypothesis. 
Hao and Ng mention in a footnote that Katsura informed them that the isomorphism is ok if $G$ is only exact.
Among other things, we'd like to understand this. 

The proof in \cite{HN} constructs an isomorphism $\OO_{X\rtimes G}\to \OO_X\rtimes G$.
On the other hand, \cite[Proposition~4.3]{kqrfunctor} asserts that for any $G$ there is a surjective homomorphism
$\OO_X\rtimes G\to \OO_{X\rtimes G}$,
and \cite[Theorem~5.1]{kqrcoact} applies this proposition to show that if $G$ is amenable, then this homomorphism is faithful, recovering the Hao-Ng result.

Unfortunately, we have discovered 
a gap
 in the proof of \cite[Proposition~4.3]{kqrfunctor}: 
it relies on
the part of \cite[Proposition~2.7]{HN} that shows that
\[
J_X\rtimes G\subset J_{X\rtimes G}\,,
\]
where $J_X$ denotes the Katsura ideal of $A$ associated with a correspondence $(X,A)$;
however,  \cite{HN} makes a blanket assumption, throughout their paper, that $G$ is amenable, and it is not at all clear that the above inclusion holds in general. However, it does hold if one considers {\it reduced} crossed products instead. As we will see in Theorem \ref{reduced}, it follows that \cite[Proposition~4.3]{kqrfunctor} is correct if it is stated in the reduced setting.

Concerning the question whether the isomorphism \eqref{iso} always holds, either for full crossed products or for reduced crossed products,
we will first look at the similar problem for Toeplitz algebras, which is easier to handle since we don't have to fuss with the Katsura ideals.
There is a technical issue --- which we are able to handle --- involving the Gauge-Invariant Uniqueness Theorem for Toeplitz algebras. The conclusion is that everything works fine in the Toeplitz case, both in the full and in the reduced setting. 

Turning to Cuntz-Pimsner algebras, we first consider the case of full crossed products and show that  \eqref{iso} holds whenever $J_X\rtimes G = J_{X\rtimes G}$. In the case of reduced crossed products, we note that D.-W.\ Kim has recently shown \cite{kim:coaction} that $\OO_X\rtimes_r G\cong \OO_{X\rtimes_r G}$ whenever a certain natural Toeplitz representation of the crossed product correspondence $(X\rtimes_r G, \, A\rtimes_r G)$ is Cuntz-Pimsner covariant. We give a different proof of this result and use it to show that $\OO_X\rtimes_r G\cong \OO_{X\rtimes_r G}$  whenever $J_X\rtimes_r G = J_{X\rtimes_r G}$.  In the case of a discrete group $G$, we show that the equality $J_X\rtimes_r G = J_{X\rtimes_r G}$ holds when $G$ is exact \cite{BO} or the action of $G$ on $A$ has Exel's approximation property \cite{exelamenable}. In particular, we get that  the isomorphism \eqref{iso}  holds in its reduced version whenever $G$ is discrete and exact.


%
%
%
%

\section{Preliminaries}\label{prelim}

Throughout this paper, $(X,A)$ will denote a nondegenerate $C^*$-correspondence, as defined in \cite{KatsuraCorrespondence}; that is, $A$ is a $C^*$-algebra and $X$ is a (right) Hilbert $A$-module with a nondegenerate left action of $A$ on $X$, given by a homomorphism $\phi=\phi_A:A\to \LL(X)$, where $\LL(X)$ denotes the $C^*$-algebra of adjointable operators on $X$.  By correspondence, we will always mean a nondegenerate $C^*$-correspondence. All homomorphisms between $C^*$-algebras will be assumed to be $\ast$-preserving, and $M(A)$ will denote the multiplier algebra of a $C^*$-algebra $A$. Nondegeneracy of the left action implies that $M(X) := \LL(A,X)$ is an $M(A)$ correspondence. We call $M(X)$ the multiplier correspondence. Our notation will be the same as in \cite{kqrfunctor}, but the reader should also consult \cite{lan:hilbert, tfb, taco, KatsuraCorrespondence, KatsuraIdeal, enchilada} for background on Hilbert $C^*$-modules and $C^*$-correspondences.  
We write: 
\begin{itemize}
\item $(t_X,t_A)$ for the universal Toeplitz representation of $(X,A)$ in the Toeplitz algebra $\TT_X$;
\item $(k_X,k_A)$ for the universal Cuntz-Pimsner covariant representation in the Cuntz-Pimsner algebra $\OO_X$;
\item $T_X:\TT_X\to \OO_X$ for the quotient map;
\item $\psi\times_\TT \pi:\TT_X\to B$ for the homomorphism associated to a Toeplitz representation $(\psi,\pi)$ of $(X,A)$ in a $C^*$-algebra $B$;
the range of $\psi\times_\TT \pi$ is then equal to the $C^*$-subalgebra $C^*(\psi, \pi)$ of $B$ generated by the ranges of $\psi$ and $\pi$; 
\item $\psi\times \pi:\OO_X\to B$ for the homomorphism associated to a Cuntz-Pimsner covariant representation $(\psi,\pi)$ of $(X,A)$ in 
$B$; clearly, we have $\psi\times_\TT \pi = (\psi\times \pi)\circ T_X$\,;
\item $J_X$ for the Katsura ideal of $A$, characterized as the largest ideal of $A$ that $\phi_A$ maps injectively into $\KK(X)$; as usual, $\KK(X)$ denotes the $C^*$-algebra of compact operators on $X$;
\item $\psi^{(1)}:\KK(X)\to B$ for the homomorphism associated to a Toeplitz representation $(\psi,\pi)$ of $(X,A)$ in $B$.
\end{itemize}

We will also assume throughout this paper that a locally compact group $G$ acts on $(X,A)$, i.e., there exists a continuous action $\gamma$ of   $G$ on $X$ that is compatible with a continuous action $\alpha$ of $G$ on $A$. The full crossed product is the completion $(X\rtimes_\gamma G,A\rtimes_\alpha G)$ of the pre-correspondence $(C_c(G,X),C_c(G,A))$ with operations
\begin{align*}
(f\cdot\xi)(s)&=\int_G f(t)\cdot\gamma_t\bigl(\xi(t\inv s)\bigr)\,dt\\
(\xi \cdot f)(s)&=\int_G \xi(t)\cdot\alpha_t\bigl(f(t^{-1}s)\bigr)\,dt\\
\inn \xi \eta (s)&=\int_G \alpha_{t^{-1}}\bigl(\bigl\<\xi(t),\eta(ts)\bigr\>\bigr)\,dt
\end{align*}
for $f,g\in C_c(G,A)$ and $\xi,\eta\in C_c(G,X)$. The reduced crossed product correspondence $(X\rtimes_{\gamma, r}G,A\rtimes_{\alpha, r}G)$ is similarly defined.
(We refer to, e.g., 
\cite{taco}, 
\cite{HN},
\cite[Chapters~2 and 3]{enchilada}, and
\cite{K}
for the elementary theory of actions and crossed products for correspondences.)

By universality of $\TT_X$ (resp.\ $\OO_X$), there is a continuous action $\tilde\beta$ (resp.\ $\beta$) of $G$ on $\TT_X$ (resp.\ $\OO_X$) determined by 
\begin{align*}
\tilde\beta_g \circ t_X &= t_X \circ \gamma_g \,, \,  \tilde\beta_g \circ t_A = t_A \circ \alpha_g\\
(\text{resp.}\quad\beta_g \circ k_X &= k_X \circ \gamma_g \,, \,  \beta_g \circ k_A = k_A \circ \alpha_g)
\end{align*}
for all $g\in G$. Clearly, we have $T_X\circ \tilde\beta_g =\beta_g \circ T_X$ for all $g\in G$; in other words, the quotient map $T_X$ is $G$-equivariant.  Whenever possible without confusing the reader, we will suppress $\gamma$, $\alpha$, $\tilde\beta$ and $\beta$ in our notation. For instance,  $(X\rtimes G, A\rtimes G)$ will denote the full crossed product correspondence, while the reduced one will be denoted by $(X\rtimes_r G, A\rtimes_r G)$. Our notation concerning $C^*$-crossed products will follow \cite{danacrossed}.  We write:
\begin{itemize}
\item $(i_A,i_G)$ for the canonical homomorphism of $(A,G)$ in the multiplier algebra $M(A\rtimes G)$;
\item $\rho\times u:A\rtimes G\to M(B)$ for the homomorphism associated to a covariant homomorphism $(\rho,u)$ of $(A,G)$ in $M(B)$ for a $C^*$-algebra $B$;
\item $(i_X,i_A)$ for the canonical homomorphism of $(X,A)$ in the crossed-product correspondence $(M(X\rtimes G),M(A\rtimes G))$;
\end{itemize}

We must apply the Gauge-Invariant Uniqueness Theorem for Toeplitz algebras
\cite[Theorem~6.2]{KatsuraCorrespondence}, which states that if $(\psi,\pi):(X,A)\to B$ is a  Toeplitz representation
that carries a gauge action, i.e.,
for every $z\in\T$ there is an endomorphism $\sigma_z$ of $C^*(\psi, \pi)$ such that
\begin{align*}
&\sigma_z(\psi(\xi))=z\, \psi(\xi)\righttext{for all}\xi\in X,\\
&\sigma_z(\pi(a))=\pi(a)\righttext{for all}a\in A,
\end{align*}
and
for which the ideal
\[
I'_{(\psi,\pi)}:=\{a\in A:\pi(a)\in \psi^{(1)}(\KK(X))\}
\]
is trivial,
then the associated homomorphism $\psi \times_\TT \pi : \TT_X\to C^*(\psi, \pi)$ is an isomorphism.
As pointed out in \cite{KatsuraCorrespondence}, $I'_{(\psi,\pi)}=\{0\}$ implies that $\pi$, and hence $\psi$, is injective,
and the other assumptions imply that in fact $\sigma$ is a continuous action of $\T$ on $B$.

We will also use that  $\OO_X$ may be characterized without using $J_X$ and the notion of Cuntz-Pimsner covariant representations (cf.\ \cite[Proposition 7.14]{KatsuraIdeal}): If $(\psi,\pi):(X,A)\to B$ is an injective Toeplitz representation
that carries a gauge action, then there exists a surjective homomorphism $\rho: C^*(\psi, \pi) \to \OO_X$ with $k_X = \rho \circ \psi, \, k_A = \rho \circ \pi$. It immediately follows that 
$T_X = \rho \circ (\psi \times_\TT \pi)$ 
in this case. 

Recall from 
\cite[Definition~A.3]{dkq} that for a nondegenerate homomorphism $\pi:A\to M(B)$
the \emph{$A$-multipliers} of $B$ are given by
\[
M_A(B)=\{m\in M(B):\pi(A)\cdot m,m\cdot \pi(A)\subset B\},
\]
and from
(a special case of) \cite[Definition~A.8]{dkq} that for a nondegenerate correspondence $(Y,C)$
the \emph{$C$-multipliers} of $Y$ are given by
\[
M_C(Y)=\{m\in M(Y):C\cdot m,m\cdot C\subset Y\}.
\]

\begin{lem}\label{extend}
Let $(\psi,\pi):(X,A)\to M(B)$ be a Toeplitz representation with $\pi$ nondegenerate.
Then $(\psi,\pi)$ extends uniquely to an $A$-strictly continuous Toeplitz representation
\[
(\bar\psi,\bar\pi):(M_A(X),M(A))\to M_A(B).
\]
In particular, $\bar\psi:M_A(X)\to M(B)$ is $A$-strict to strictly continuous.
\end{lem}

\begin{proof}
The first statement is a very special case of \cite[Corollary~A.13]{dkq}, and the second then follows by \cite[Corollary~A.4 (1)]{dkq}.
\end{proof}

\begin{notn}
We will have to extend Toeplitz representations using \lemref{extend} quite often,
and to clean up the notation we will suppress the ``bar'', i.e., we will continue to write $(\psi,\pi)$ for the canonical extension to $(M_A(X),M(A))$.
\end{notn}

\begin{cor}\label{compose}
Let $(X,A)$ and $(Y,B)$ be nondegenerate correspondences, 
and let $(\psi,\pi):(X,A)\to (M(Y),M(B))$ be a
correspondence homomorphism,
and let $(\sigma,\tau):(Y,B)\to M(C)$ be a Toeplitz representation.
Assume that
$\psi(X)\subset M_B(Y)$,
and that
$\pi$ and $\tau$ are nondegenerate.
Then the composition
\[
(\sigma\circ\psi,\tau\circ\pi):(X,A)\to M(C)
\]
is a Toeplitz representation,
and the associated homomorphism
\[
(\sigma\circ\psi)\times_\TT (\tau\circ\pi):\TT_X\to M(C)
\]
is nondegenerate.
\end{cor}

\begin{proof}
The hypotheses allow us to apply \lemref{extend} to get a Toeplitz representation
\[
(\sigma,\tau):(M_B(Y),M(B))\to M(C),
\]
and it is routine to verify that the composition is a Toeplitz representation.
The nondegeneracy follows since $\tau\circ\pi:A\to M(C)$ is nondegenerate.
\end{proof}

\begin{cor}\label{generate toeplitz}
Let $(X,A)$ and $(Y,B)$ be nondegenerate correspondences, 
and let $(\psi,\pi):(X,A)\to (M(Y),M(B))$ be a
correspondence homomorphism.
Assume that
$\psi(X)\subset M_B(Y)$
and
$\pi$ is nondegenerate.
Then the composition
\[
(t_Y\circ\psi,t_B\circ\pi):(X,A)\to M(\TT_Y)
\]
is a Toeplitz representation,
and the associated homomorphism
\[
(t_Y\circ\psi)\times_\TT (t_B\circ\pi):\TT_X\to M(\TT_Y)
\]
is nondegenerate,
and is faithful if $\pi$ is.
\end{cor}

\begin{proof}
The first two statements follow immediately from \corref{compose},
so assume that $\pi$ is injective.
To show that $(t_Y\circ\psi)\times_\TT (t_B\circ\pi)$ is injective,
our aim is to apply the Gauge-Invariant Uniqueness Theorem \cite[Theorem~6.2]{KatsuraCorrespondence}.
Since $\TT_Y$ has a gauge action, it quickly follows that $(t_Y\circ\psi,t_B\circ\pi)$ carries a gauge action.
For the other part, let $a\in A$, and assume that
$t_B\circ\pi(a)\in (t_Y\circ\psi)^{(1)}(\KK(X))$.
Then there exists $k\in\KK(X)$ such that
for all $b\in B$,
\begin{align*}
t_B(b\pi(a))
&=t_B(b)t_B(\pi(a))
\\&=t_B(b)t_Y^{(1)}(\psi^{(1)}(k))
\\&=t_Y^{(1)}\bigl(\varphi_B(b)\psi^{(1)}(k)\bigr)
\\&\in t_Y^{(1)}(\KK(Y)),
\end{align*}
where the last step follows since the hypotheses imply that
$\psi^{(1)}(k)\in M(\KK(Y))$.
Since
$I'_{(t_Y,t_B)}=\{0\}$
by \cite[Corollary~4.5]{KatsuraCorrespondence},
it follows that
$b\pi(a)=0$ for all $b\in B$, so $\pi(a)=0$, and hence $a=0$ since $\pi$ is injective.
\end{proof}

\section{Toeplitz crossed products}\label{toep sec}

We show in this section that everything works as it should in the Toeplitz case: the crossed product of the Toeplitz algebra is the Toeplitz algebra of the crossed product correspondence, both for full and reduced crossed products.

Since the correspondence homomorphism $(i_X,i_A):(X,A)\to (M(X\rtimes_\gamma G),M(A\rtimes_\alpha G))$ 
satisfies
$i_X(X)\subset M_{A\rtimes G}(X\rtimes G)$
and
$i_A$ is nondegenerate,
by \corref{generate toeplitz}
the composition
\[
(t_{X\rtimes G}\circ i_X,t_{A\rtimes G}\circ i_A):(X,A)\to M(\TT_{X\rtimes G})
\]
is a Toeplitz representation,
and
\[
(t_{X\rtimes G}\circ i_X)\times_\TT (t_{A\rtimes G}\circ i_A)
:\TT_X\to M(\TT_{X\rtimes G})
\]
is nondegenerate (and faithful).

Then
computations similar to those in the proof of \cite[Proposition~4.2]{kqrfunctor}
show that the pair
\[
\bigl((t_{X\rtimes G}\circ i_X)\times_\TT (t_{A\rtimes G}\circ i_A),t_{A\rtimes G}\circ i_G\bigr)
\]
is
a covariant homomorphism of $(\TT_X,G)$
in $M(\TT_{X\rtimes G})$.

\begin{thm}\label{toep full}
The integrated form
\begin{multline*}
\Phi:=
\bigl((t_{X\rtimes G}\circ i_X)\times_\TT (t_{A\rtimes G}\circ i_A)\bigr)\times
(t_{A\rtimes G}\circ i_G):
\\
\TT_X\rtimes G\to M(\TT_{X\rtimes G})
\end{multline*}
of the above covariant pair is an isomorphism onto $\TT_{X\rtimes G}$.
\end{thm}

\begin{proof}
Since 
$(t_{X\rtimes G}\circ i_X)\times_\TT (t_{A\rtimes G}\circ i_A)$
is nondegenerate, so is $\Phi$.
We will construct an inverse for $\Phi$.

The universal Toeplitz representation $(t_X,t_A):(X,A)\to \TT_X$ is $G$-equivariant,
so the crossed product gives a homomorphism
\[
\Theta:=(t_X\rtimes G)\times_\TT (t_A\rtimes G):\TT_{X\rtimes G}\to M(\TT_X\rtimes G)
\]
that is nondegenerate because $t_A\rtimes G$ is.

For $x\in X$ we have
\begin{align*}
\Phi\circ \Theta\circ t_{X\rtimes G}\circ i_X(x)
&=\Phi\circ (t_X\rtimes G)\circ i_X(x)
\\&=\Phi\circ i_{\TT_X}\circ t_X(x)
\\&=\bigl((t_{X\rtimes G}\circ i_X)\times_\TT (t_{A\rtimes G}\circ i_A)\bigr)\circ t_X(x)
\\&=t_{X\rtimes G}\circ i_X(x),
\end{align*}
for $a\in A$ we have
\begin{align*}
\Phi\circ \Theta\circ t_{A\rtimes G}\circ i_A(a)
&=\Phi\circ (t_A\rtimes G)\circ i_A(a)
\\&=\Phi\circ i_{\TT_X}\circ t_A(a)
\\&=\bigl((t_{X\rtimes G}\circ i_X)\times_\TT (t_{A\rtimes G}\circ i_A)\bigr)\circ t_A(a)
\\&=t_{A\rtimes G}\circ i_A(a),
\end{align*}
and for $s\in G$ we have
\begin{align*}
\Phi\circ \Theta\circ t_{A\rtimes G}\circ i_G^A(s)
&=\Phi\circ (t_A\rtimes G)\circ i_G^A(s)
\\&=\Phi\circ i_G^{\TT_X}(s)
\\&=t_{A\rtimes G}\circ i_G^A(s).
\end{align*}
It follows that
$\Phi\circ\Theta$ is the identity on $\TT_{X\rtimes G}$.

On the other hand,
for $x\in X$ we have
\begin{align*}
\Theta\circ \Phi\circ i_{\TT_X}\circ t_X(x)
&=\Theta\circ \bigl((t_{X\rtimes G}\circ i_X)\times_\TT (t_{A\rtimes G}\circ i_A)\bigr)\circ t_X(x)
\\&=\Theta\circ t_{X\rtimes G}\circ i_X(x)
\\&=(t_X\rtimes G)\circ i_X(x)
\\&=i_{\TT_X}\circ t_X(x),
\end{align*}
for $a\in A$ we have
\begin{align*}
\Theta\circ \Phi\circ i_{\TT_X}\circ t_A(a)
&=\Theta\circ \bigl((t_{X\rtimes G}\circ i_X)\times_\TT (t_{A\rtimes G}\circ i_A)\bigr)\circ t_A(a)
\\&=\Theta\circ t_{A\rtimes G}\circ i_A(a)
\\&=(t_A\rtimes G)\circ i_A(a)
\\&=i_{\TT_X}\circ t_A(a),
\end{align*}
and for $s\in G$ we have
\begin{align*}
\Theta\circ \Phi\circ i_G^{\TT_X}(s)
&=\Theta\circ t_{A\rtimes G}\circ i_G^A(s)
\\&=(t_A\rtimes G)\circ i_G^A(s)
\\&=i_G^{\TT_X}(s).
\end{align*}
It follows that $\Theta\circ \Phi$ is the identity on $\TT_X\rtimes G$,
and we have now shown that $\Theta$ is an inverse for $\Phi$.
\end{proof}

\begin{thm}\label{toep red}
There is a unique isomorphism $\Phi_r$ making the diagram
\[
\xymatrix@C+20pt{
\TT_X\rtimes G \ar[r]^-\Phi_-\simeq \ar[d]_{\Lambda_{\TT_X}}
&\TT_{X\rtimes G} 
\ar[d]^{(t_{X\rtimes_r G}\circ \Lambda_X)\times_\TT (t_{A\rtimes_r G}\circ \Lambda_A)}
\\
\TT_X\rtimes_r G \ar@{-->}[r]_-{\Phi_r}^-\simeq
&\TT_{X\rtimes_r G}
}
\]
commute,
where $\Phi$ is the isomorphism from \thmref{toep full}.
\end{thm}

\begin{proof}
To clarify the notation in the above diagram,
the left-hand vertical map is the regular representation of the crossed product of $\TT_X$,
while the right-hand vertical map is the canonical homomorphism
associated to the Toeplitz representation
\[
(t_{X\rtimes_r G}\circ \Lambda_X,t_{A\rtimes_r G}\circ \Lambda_A)
\]
of the crossed-product correspondence $(X\rtimes G,A\rtimes G)$
in the Toeplitz algebra $\TT_{X\rtimes_r G}$,
where in turn
\[
(\Lambda_X,\Lambda_A):(X\rtimes G,A\rtimes G)\to (X\rtimes_r G,A\rtimes_r G)
\]
is the regular representation of the crossed-product correspondence.

It seems difficult to construct $\Phi_r$ directly,
because it is not easy to decide whether 
the composition
\[
\TT_X\rtimes G\to \TT_{X\rtimes G}\to \TT_{X\rtimes_r G}
\]
factors through the reduced crossed product $\TT_X\rtimes_r G$.
But we can get a map going the other way:
since the universal Toeplitz representation
\[
(t_X,t_A)\to \TT_X
\]
is $G$-equivariant, we get a Toeplitz representation
\[
(t_X\rtimes_r G,t_A\rtimes_r G):(X\rtimes_r G,A\rtimes_r G)\to \TT_X\rtimes_r G
\]
that fits into a commutative diagram
\[
\xymatrix@C+60pt{
\TT_X\rtimes G \ar[d]_{\Lambda_{\TT_X}}
&(X\rtimes G,A\rtimes G) \ar[l]_-{(t_X\rtimes G,t_A\rtimes G)} \ar[d]^{(\Lambda_X,\Lambda_A)}
\\
\TT_X\rtimes_r G
&(X\rtimes_r G,A\rtimes_r G). \ar[l]^-{(t_X\rtimes_r G,t_A\rtimes_r G)}
}
\]
Passing to the canonical homomorphisms associated to the horizontal Toeplitz representations gives the commutative diagram
\[
\xymatrix@C+90pt{
\TT_X\rtimes G \ar[d]_{\Lambda_{\TT_X}}
&\TT_{X\rtimes G} \ar[l]_-{\Theta=(t_X\rtimes G)\times_\TT (t_A\rtimes G)}^-\simeq
\ar[d]^{(t_{X\rtimes_r G}\circ \Lambda_X)\rtimes_\TT (t_{A\rtimes_r G}\circ \Lambda_A)}
\\
\TT_X\rtimes_r G
&\TT_{X\rtimes_r G}, \ar[l]^-{\Theta_r:=(t_X\rtimes_r G)\times_\TT (t_A\rtimes_r G)}
}
\]
where the $\Theta=\Phi\inv$ at the top was introduced in the proof of \thmref{toep full}.

We will show that $\Theta_r$ is an isomorphism, and then its inverse will be the desired $\Phi_r$.
We see that $\Theta_r$ is surjective, because generators go to generators:
for $x\in X$ we have
\begin{align*}
\Theta_r\circ t_{X\rtimes_r G}\circ i_X^r(x)
&=(t_X\rtimes_r G)\circ i_X^r(x)
=i_{T_X}^r\circ t_X(x),
\end{align*}
for $a\in A$ we have
\begin{align*}
\Theta_r\circ t_{A\rtimes_r G}\circ i_A^r(a)
&=(t_A\rtimes_r G)\circ i_A^r(a)
=i_{T_X}^r\circ t_A(a),
\end{align*}
and for $s\in G$ we have
\begin{align*}
\Theta_r\circ t_{A\rtimes_r G}\circ i_G^{A,r}(s)
&=(t_A\rtimes_r G)\circ i_G^{A,r}(s)
=i_G^{\TT_X,r}(s).
\end{align*}

We will show that $\Theta_r$ is injective, and we aim to apply the Gauge-Invariant Uniqueness Theorem for Toeplitz algebras
\cite[Theorem~6.2]{KatsuraCorrespondence}.
Thus, for every $z\in\T$ we need an endomorphism $\sigma'_z$ of $\TT_X\rtimes_r G$ such that
\begin{align*}
\sigma'_z\circ (t_X\rtimes_r G)&=z\,(t_X\rtimes_r G)\\
\sigma'_z\circ (t_A\rtimes_r G)&=t_A\rtimes_r G.
\end{align*}
Let $\sigma$ be the gauge action on $\TT_X$.
Since this action commutes with the action of $G$,
it induces an action $\sigma'=\sigma\rtimes_r G$ on the reduced crossed product $\TT_X\rtimes_r G$,
with
\[
\sigma'_z=\sigma_z\rtimes_r G\righttext{for}z\in \T.
\]

Let $z\in\T$.
For $\xi\in C_c(G,X)$ we have
\begin{align*}
(\sigma_z\rtimes_r G)\circ (t_X\rtimes_r G)(\xi)
&=\bigl((\sigma_z\circ t_X)\rtimes_r G\bigr)(\xi)
\\&=\sigma_z\circ t_X\circ\xi
\\&=z\,(t_X\circ\xi)
\\&=z\,(t_X\rtimes_r G)(\xi),
\end{align*}
and for $f\in C_c(G,A)$ we have
\begin{align*}
(\sigma_z\rtimes_r G)\circ (t_A\rtimes_r G)(f)
&=\bigl((\sigma_z\circ t_A)\rtimes_r G\bigr)(f)
\\&=\sigma_z\circ t_A\circ f
\\&=t_A\circ f
\\&=(t_A\rtimes_r G)(f).
\end{align*}

Finally, we need to show that the ideal
\begin{multline*}
I'_{(t_X\rtimes_r G,t_A\rtimes_r G)}
\\=\{b\in A\rtimes_r G:(t_A\rtimes_r G)(b)\in (t_X\rtimes_r G)^{(1)}(\KK(X\rtimes_r G)\}
\end{multline*}
of $A\rtimes_r G$ is trivial.
But this follows from \corref{corona} below,
since
$I'_{(t_X,t_A)}=\{0\}$ by \cite[Corollary~4.5]{KatsuraCorrespondence}.
\end{proof}

\begin{rem}
We could have applied the strategy of the above proof to prove \thmref{toep full}, but for full crossed products we were able to find the inverse homomorphism, and we feel that this leads to a more elementary proof.
\end{rem}

\begin{lem}\label{quotient}
Let $(A,\alpha)$ and $(C,\sigma)$ be actions of $G$,
let $\pi:A\to C$ be an $\alpha-\sigma$ equivariant homomorphism,
and let $K$ be a $\sigma$-invariant ideal of $C$.
Denote by 
$q:C\to Q:=C/K$ the quotient map and by
$\tau$ the quotient action of $G$ on $Q$.
If $\ker (q\circ \pi)=\{0\}$,
then $\ker ((q\circ \pi)\rtimes_r G)=\{0\}$.
In particular, if $\ker (q\circ \pi)=\{0\}$ and $b\in A\rtimes_{\alpha,r} G$ satisfies $(\pi\rtimes_r G)(b)\in K\rtimes_{\sigma,r} G$, then $b=0$.
\end{lem}

\begin{proof}
The first statement is a well-known property of the reduced-crossed-product functor.
For the second, just note that
\[
(q\circ \pi)\rtimes_r G=(q\rtimes_r G)\circ (\pi\times_r G)
\]
by functoriality, and
\[
K\rtimes_{\sigma,r} G\subset \ker (q\rtimes_r G)
\]
since $K\subset \ker q$.
\end{proof}

The following result was used in the proof of \thmref{toep red} above;
we give a general formulation since it might be useful elsewhere.

\begin{cor}\label{corona}
Let $(\psi,\pi):(X,A)\to M(D)$ be a Toeplitz representation
that is equivariant for some action 
$\sigma$ 
of $G$ on $D$.
If $I'_{(\psi,\pi)}=\{0\}$, then
$I'_{(\psi\rtimes_r G,\pi\rtimes_r G)}=\{0\}$.
\end{cor}

\begin{proof}
We aim to apply \lemref{quotient} with $K=\psi^{(1)}(\KK(X))$
and $C$ equal to the $C^*$-subalgebra of $M(D)$ generated by $\pi(A)\cup K$.
By \cite[Lemma~2.4]{KatsuraCorrespondence}, $\pi(A)$ idealizes $K$,
so $K$ is an ideal of $C$.

Since $\psi^{(1)}:\KK(X)\to M(D)$ is $\gamma^{(1)}-\sigma$ equivariant, 
$K$ is a $\sigma$-invariant $C^*$-subalgebra of $M(D)$,
and $\sigma$ is strongly continuous on $K$ since the action $\gamma^{(1)}$ is strongly continuous.
Similarly,  $\sigma$ is strongly continuous on $\pi(A)$.
Thus $\sigma$ is strongly continuous on the $*$-subalgebra of $M(D)$ generated by $\pi(A)\cup K$,
and hence on $C$ since a uniform limit of continuous functions is continuous.

With the notation of \lemref{quotient},
the hypothesis is that $\ker (q\circ \pi)=\{0\}$.
Let $b\in I'_{(\psi\rtimes_r G,\pi\rtimes_r G)}$.
Then
\[
(\pi\rtimes_r G)(b)\in (\psi\rtimes_r G)^{(1)}\bigl(\KK(X\rtimes_r G)\bigr).
\]

Using the isomorphism $\KK(X\rtimes_r G)\cong \KK(X)\rtimes_r G$ as in the diagram
\[
\xymatrix@C+20pt{
\KK(X\rtimes_r G) \ar[r]^-{(\psi\rtimes_r G)^{(1)}} \ar[d]_\simeq
&M(D\rtimes_{\sigma,r} G)
\\
\KK(X)\rtimes_{\gamma^{(1)},r} G\ar[r]_-{\psi^{(1)}\rtimes_r G}
&K\rtimes_{\sigma,r} G, \ar[u]_\subset
}
\]
we get
\[
(\pi\rtimes_r G)(b)\in K\rtimes_{\sigma,r} G\subset \ker (q\rtimes_r G),
\]
and hence
\[
b\in \ker (q\rtimes_r G)\circ (\pi\rtimes_r G)=\ker \bigl((q\circ \pi)\rtimes_r G\bigr),
\]
so by \lemref{quotient} we get $b=0$.
\end{proof}

A $C^*$-dynamical system $(A,G)$ is often called \emph{regular} when the left regular representation $\Lambda_A : A\rtimes G \to  A\rtimes_r G$ is injective. This means that the universal norm on $C_c(G,A)$ coincides with the reduced one, so if one regards $A\rtimes G$ and $ A\rtimes_r G$ as completions of $C_c(G, A)$ with respect to these norms, we have $A\rtimes G = A\rtimes_r G$. We therefore often write $A\rtimes G = A\rtimes_r G$ instead of saying that $(A,G)$ is regular. As is well known, see e.g.\  \cite{danacrossed}, $(A, G)$ is regular whenever $G$ is amenable. More generally, $(A, G)$ is regular whenever the action of $G$ on $A$ has Exel's approximation property \cite{exelamenable, exelng}. For other conditions ensuring regularity, we refer to \cite{ana, anan:exact, qs:regularity, BO, bc:regular}. 
  
\begin{thm}\label{toep regular}
Suppose that 
$A\rtimes G=A\rtimes_r G$.
Then 
also
\begin{align*}
X\rtimes G&=X\rtimes_r G\righttext{and}\\
\TT_X \rtimes G &=  \TT_X \rtimes_r G \simeq \TT_{X \rtimes_r G} = \TT_{X \rtimes G} \,.
\end{align*}
\end{thm}

\begin{proof}
The regular representation $\Lambda_X$ is injective because its right-coefficient homomorphism $\Lambda_A$ is.
For the other part, consider the commutative diagram
\[
\xymatrix@C+20pt{
\TT_X\rtimes G \ar[r]^-\Phi_-\simeq \ar[d]_{\Lambda_{\TT_X}}
&\TT_{X\rtimes G} 
\ar[d]^{(t_{X\rtimes_r G}\circ \Lambda_X)\times_\TT (t_{A\rtimes_r G}\circ \Lambda_A)}_\simeq
\\
\TT_X\rtimes_r G \ar@{-->}[r]_-{\Phi_r}^-\simeq
&\TT_{X\rtimes_r G}
}
\]
from \thmref{toep red}.
The left-hand vertical map $\Lambda_{\TT_X}$ must be an isomorphism since the other three maps are.
\end{proof}

We include one application concerning nuclearity:
\begin{prop} Assume that the action of $G$ on $A$ has Exel's approximation property.
Then we have
$$\TT_X \rtimes G =  \TT_X \rtimes_r G \simeq \TT_{X \rtimes_r G} = \TT_{X \rtimes G} $$
and these $C^*$-algebras are nuclear whenever $A$ is nuclear.
\end{prop}

\begin{proof} This follows immediately by combining Theorem \ref{toep regular} and \cite[Theorems 3.9 and 4.4]{exelng}.
\end{proof}

We also include a result concerning exactness. We recall that  $G$ is called {\it exact} (sometimes called {KW-exact}) if for every short exact sequence 
$0 \rightarrow I \rightarrow B \rightarrow B/I \rightarrow 0$
of $G$-$C^*$-algebras, 
the induced sequence 
$$0 \rightarrow I \rtimes _r G \rightarrow B \rtimes _r G \rightarrow (B/I) \rtimes _r G \rightarrow 0$$
is also exact. It is known that $C_r^*(G)$ is exact as a $C^*$-algebra whenever $G$ is exact, and that the converse also holds if $G$ is discrete. (See \cite[Section 5.1]{BO} and references therein).

%
%

\begin{prop} \label{toep exact}
Consider the following conditions:
 \begin{enumerate}
 \item $G$ is exact and $A$ is exact;
 \item $\TT_{X} \rtimes_r G$ is exact;
 \item $\TT_{X \rtimes_r G}$ is exact.
 \end{enumerate}
 Then we have \textup{(1)} $\Rightarrow$ \textup{(2)} $\Leftrightarrow$ \textup{(3).} 
 If $G$ is discrete, then all three conditions are equivalent.
 \end{prop}
 \begin{proof}
 Assume that (1) holds. Then, by \cite[Theorem 7.1]{KatsuraCorrespondence}, $\TT_X$ is exact,  so,
 \cite[Theorem 7.2]{anan:exact} gives that (2) holds. By Theorem \ref{toep red}, $\TT_{X \rtimes_r G} \simeq  \TT_{X} \rtimes_r G$, so (2) is equivalent to (3). Assume now that $G$ is discrete and (3) holds. Then, by \cite[Theorem 7.1]{KatsuraCorrespondence}, $A\rtimes_r G$ is exact, so (1) holds because  exactness passes to $C^*$-subalgebras and to unitizations. 
\end{proof}

\section{Hao-Ng for full crossed products}

In this section we prove a version of the Hao-Ng theorem for full crossed products.

\begin{thm}\label{full}
If $J_X\rtimes G=J_{X\rtimes G}$\,,
then
there is a unique isomorphism $\Psi$ making the diagram
\[
\xymatrix@C+50pt{
\TT_X\rtimes G \ar[r]^-\Phi_-\simeq \ar[d]_{T_X\rtimes G}
&\TT_{X\rtimes G} \ar[d]^{T_{X\rtimes G}}
\\
\OO_X\rtimes G \ar@{-->}[r]_-\Psi^-{\simeq}
&\OO_{X\rtimes G}
}
\]
commute,
where $\Phi$ is the isomorphism from \thmref{toep full}.
\end{thm}

We feel that it is instructive to split the result into two halves:

\begin{lem}\label{contain 1}
If $J_X\rtimes G\subset J_{X\rtimes G}$\, , then
there is a unique homomorphism $\Psi$ making the diagram
\begin{equation}\label{Psi}
\xymatrix@C+50pt{
\TT_X\rtimes G \ar[r]^-\Phi \ar[d]_{T_X\rtimes G}
&\TT_{X\rtimes G} \ar[d]^{T_{X\rtimes G}}
\\
\OO_X\rtimes G \ar@{-->}[r]_-\Psi^-{!}
&\OO_{X\rtimes G}
}
\end{equation}
commute.
Moreover, $\Psi$ is surjective.
\end{lem}

\begin{proof}
Of course, once we know $\Psi$ exists,
it is unique since $T_X\rtimes G$ is surjective, and
is surjective since $\Phi$ and $T_{X\rtimes G}$ are.
For the existence, 
we first claim that the Toeplitz representation
\[
(k_{X\rtimes G}\circ i_X,k_{A\rtimes G}\circ i_A):(X,A)\to M(\OO_{X\rtimes G})
\]
is Cuntz-Pimsner covariant.
Since the correspondence homomorphism $(i_X,i_A)$ is nondegenerate,
by \cite[Lemma~3.2]{kqrfunctor} it suffices to show:
\begin{enumerate}
\renewcommand{\labelenumi}{(\roman{enumi})}
\item $i_X(X)\subset M_{A\rtimes G}(X\rtimes G)$;
\item $i_A(J_X)\subset M(A\rtimes G;J_{X\rtimes G})$
\end{enumerate}
where 
\[
 M(A\rtimes G;J_{X\rtimes G})= \{m\in M(A\rtimes G) : m(A\rtimes G) \cup (A\rtimes G)m \subset J_{X\rtimes G} \}.
\]
For (i), 
it is enough to observe that if $x\in X$ and $f\in C_c(G,A)\subset A\rtimes G$ then $f\cdot i_X(x)$ is the element of $C_c(G,x)$ given by
\[
(f\cdot i_X(x))(s)=f(s)\cdot \gamma_s(x),
\]
and to see this we compute that,
for $g\in C_c(G,A)$,
\begin{align*}
\bigl((f\cdot i_X(x))\cdot g\bigr)(s)
&=\bigl(f\cdot (i_X(x)\cdot g)\bigr)(s)
\\&=\int f(t)\cdot \gamma_t\bigl((i_X(x)\cdot g)(t\inv s)\bigr)\,dt
\\&=\int f(t)\cdot \gamma_t\bigl(x\cdot g(t\inv s)\bigr)\,dt
\\&=\int f(t)\cdot \gamma_t(x)\cdot \alpha_t(g(t\inv s))\,dt.
\end{align*}

For (ii), we need to know that $i_A(J_X)$ multiplies $A\rtimes G$ into the Katsura ideal $J_{X\rtimes G}$.
For $a\in J_X$ and $f\in C_c(G,A)$ we have
\[
\bigl(i_A(a)f\bigr)(s)=af(s)\in J_X,
\]
so
\[
i_A(a)f\in C_c(G,J_X)\subset J_X\rtimes G\subset J_{X\rtimes G},
\]
and it follows that $i_A(a)$ multiplies $A\rtimes G$ into $J_{X\rtimes G}$, as desired.

Thus
we have
a homomorphism
\[
(k_{X\rtimes G}\circ i_X)\times (k_{A\rtimes G}\circ i_A):\OO_X\to M(\OO_{X\rtimes G}).
\]
Then
the argument in
\cite[proof of Proposition~4.3]{kqrfunctor}
gives a surjective homomorphism
\[
\Psi:=\bigl((k_{X\rtimes G}\circ i_X)\times (k_{A\rtimes G}\circ i_A)\bigr)\times (k_{A\rtimes G}\circ i_G):\OO_X\rtimes G\to \OO_{X\rtimes G}.
\]

Finally, recalling from \thmref{toep full} that
\[
\Phi=\bigl((t_{X\rtimes G}\circ i_X)\times_\TT (t_{A\rtimes G}\circ i_A)\bigr)\times (
t_{A\rtimes G}\circ i_G),
\]
it is routine to check that the diagram~\eqref{Psi} commutes.
\end{proof}

\begin{lem}\label{contain 2}
If $J_X\rtimes G\supset J_{X\rtimes G}$, then
there is a unique homomorphism $\Upsilon$ making the diagram
\begin{equation}\label{surjection 2}
\xymatrix@C+50pt{
\TT_X\rtimes G\ar[d]_{T_X\rtimes G}
&\TT_{X\rtimes G} \ar[d]^{T_{X\rtimes G}} \ar[l]_-{\Phi\inv}
\\
\OO_X\rtimes G
&\OO_{X\rtimes G} \ar@{-->}[l]^-\Upsilon_-{!}
}
\end{equation}
commute.
Moreover, $\Upsilon$ is surjective.
\end{lem}

\begin{proof}
Again, once we know $\Upsilon$ exists,
it is unique since $T_{X\rtimes G}$ is surjective, and
is surjective since $\Phi\inv$ and $T_X\rtimes G$ are.
For the existence, first note that the commutative diagram
\[
\xymatrix@C+20pt{
\TT_X \ar[d]_{T_X}
&(X,A) \ar[l]_-{(t_X,t_A)} \ar[dl]^{(k_X,k_A)}
\\
\OO_X
}
\]
is $G$-equivariant, and taking crossed products gives the commutative diagram
\begin{equation}\label{triangle}
\xymatrix@C+40pt{
\TT_X\rtimes G \ar[d]_{T_X\rtimes G}
&(X\rtimes G,A\rtimes G) \ar[l]_-{(t_X\rtimes G,t_A\rtimes G)} \ar[dl]^{(k_X\rtimes G,k_A\rtimes G)}
\\
\OO_X\rtimes G.
}
\end{equation}
We claim that the Toeplitz representation $(k_X\rtimes G,k_A\rtimes G)$
is Cuntz-Pimsner covariant.
We must show that
\[
(k_X\rtimes G)^{(1)}\circ \phi_{A\rtimes G}=k_A\rtimes G
\]
on the Katsura ideal $J_{X\rtimes G}$.
Again employing the isomorphism $\KK(X\rtimes G)\cong \KK(X)\rtimes G$,
since the diagram
\[
\xymatrix@C+20pt{
J_X\rtimes G \ar[r]^-{\phi_{A\rtimes G}} \ar[dr]_{\phi_A\rtimes G}
&\KK(X\rtimes G) \ar[r]^-{(k_X\rtimes G)^{(1)}} \ar[d]_\simeq
&\OO_X\rtimes G
\\
&\KK(X)\rtimes G \ar[ur]_{k_X^{(1)}\rtimes G}
}
\]
commutes,
it suffices to show that
\begin{equation}\label{equal}
(k_X^{(1)}\rtimes G)\circ (\phi_A\rtimes G)=k_A\rtimes G
\end{equation}
on $J_{X\rtimes G}$.
Now, since the equality $k_X^{(1)}\circ \phi_A=k_A$ on $J_X$ is $G$-equivariant,
taking crossed products gives
\begin{equation}\label{equal2}
(k_X^{(1)}\circ \phi_A)\rtimes G=k_A\rtimes G\midtext{on}J_X\rtimes G.
\end{equation}
We have $J_{X\rtimes G}\subset J_X\rtimes G$ by hypothesis,
and
\[
(k_X^{(1)}\rtimes G)\circ (\phi_A\rtimes G)=(k_X^{(1)}\circ \phi_A)\rtimes G,
\]
so \eqref{equal} follows from \eqref{equal2}.

Thus we have a homomorphism
\[
\Upsilon:=(k_X\rtimes G)\times (k_A\rtimes G):\OO_{X\rtimes G}\to \OO_X\rtimes G,
\]
which makes the diagram \eqref{surjection 2} commute
because \eqref{triangle} commutes.
\end{proof}

\begin{proof}[Proof of \thmref{full}]
This follows immediately from Lemmas~\ref{contain 1} and \ref{contain 2},
since the rectangles commute and the vertical maps are surjections.
\end{proof}

\section{Hao-Ng for reduced crossed products}

In this section 
we discuss versions of the Hao-Ng theorem \cite[Theorem~2.10]{HN} for reduced crossed products. 

We first note that the commutative diagram
\[
\xymatrix@C+20pt{
\TT_X \ar[d]_{T_X}
&(X,A) \ar[l]_-{(t_X,t_A)} \ar[dl]^{(k_X,k_A)}
\\
\OO_X
}
\]
is $G$-equivariant, and taking reduced crossed products gives the commutative diagram
\begin{equation}\label{triangle red}
\xymatrix@C+40pt{
\TT_X\rtimes_r G \ar[d]_{T_X\rtimes_r G}
&(X\rtimes_r G,A\rtimes_r G) \ar[l]_-{(t_X\rtimes_r G,t_A\rtimes_r G)} \ar[dl]^{(k_X\rtimes_r G,k_A\rtimes_r G)}
\\
\OO_X\rtimes_r G
}
\end{equation}
In a recent paper \cite{kim:coaction} D.-W.\ Kim deduces the following result from a more general result dealing with coactions of Hopf $C^*$-algebras on $C^*$-correspondences:  

\begin{thm}
[{\cite[Corollary 5.11]{kim:coaction}}]
\label{kim}
If  $(k_X\rtimes_r G,\,k_A\rtimes_r G)$ is Cuntz-Pimsner covariant, then $\Upsilon_r:=(k_X\rtimes_r G) \times (k_A\rtimes_r G)$ is an isomorphism from $\OO_{X \rtimes_r G}$ onto  $\OO_X \rtimes_r G$.
\end{thm}  
Kim also provides some equivalent conditions for $(k_X\rtimes_r G,\,k_A\rtimes_r G)$  being Cuntz-Pimsner covariant, which are satisfied for instance when $J_X =A$ or when $\phi_A$ is injective; see \cite[Theorem 5.7 and Corollary 5.8]{kim:coaction}.

\smallskip Theorem \ref{kim} may be shown by using the Gauge-Invariant Uniqueness Theorem for Cuntz-Pimsner algebras, essentially as in the proof of \cite[Theorem~2.10]{HN}. We provide an alternative approach:

\begin{thm}\label{reduced}
\begin{enumerate}
\item
There is a unique homomorphism $\Psi_r$ making the diagram
\begin{equation}\label{Psi r}
\xymatrix@C+50pt{
\TT_X\rtimes_r G \ar[r]^-{\Phi_r} \ar[d]_{T_X\rtimes_r G}
&\TT_{X\rtimes_r G} \ar[d]^{T_{X\rtimes_r G}}
\\
\OO_X\rtimes_r G \ar@{-->}[r]_-{\Psi_r}^-{!}
&\OO_{X\rtimes_r G}
}
\end{equation}
commute, where $\Phi_r$ is the isomorphism from \thmref{toep red}.
The map $\Psi_r$ is surjective and satisfies $$\Psi_r \circ (k_X \rtimes_r G) = k_{X \rtimes_r G}\, , \,  \Psi_r \circ (k_A\rtimes_r G) = k_{A \rtimes_r G}\,.$$

\item
Assume that $(k_X \rtimes_r  G,\,k_A\rtimes_r  G)$ is Cuntz-Pimsner covariant. Then $\Psi_r$ is an isomorphism, and its inverse is $\Upsilon_r = (k_X\rtimes_r G) \times (k_A\rtimes_r G)$. The isomorphism $\Upsilon_r$ is the unique homomorphism  making the diagram
\begin{equation}\label{surjection 2 red}
\xymatrix@C+50pt{
\TT_X\rtimes_r G\ar[d]_{T_X\rtimes_r G}
&\TT_{X\rtimes_r G} \ar[d]^{T_{X\rtimes_r G}} \ar[l]_-{(\Phi_r)\inv}^-\simeq
\\
\OO_X\rtimes_r G
&\OO_{X\rtimes_r G} \ar@{-->}[l]^-{\Upsilon_r}_-{!}
}
\end{equation}
commute.
\end{enumerate}
\end{thm} 
  
%
%
%

\begin{proof}
(1)
We consider  the Toeplitz representation $(k_X\rtimes_r G,\,k_A\rtimes_r G)$ of $(X\rtimes_r G, \, A\rtimes_r G)$ in $\OO_X \rtimes_r G$. As $k_A$ is injective, $k_A\rtimes_r G$ is also injective, by \cite[Lemma~3.1]{lan:dual}. Hence, $(k_X\rtimes_r G,\,k_A\rtimes_r G)$ is injective. Moreover, it carries a gauge action: the proof of this fact is similar to that for $\Theta_r$ in the proof of \thmref{toep red},
using $(k_X,k_A)$ instead of $(t_X,t_A)$. Now, as recalled in the Preliminaries, the Cuntz-Pimsner algebra of a correspondence is the smallest $C^*$-algebra generated by an injective Toeplitz representation that carries a gauge action (cf.\ \cite[Proposition 7.14]{KatsuraIdeal}). Hence, there exists a homomorphism $\Psi_r $ from $  C^*(k_X\rtimes_r G,\,k_A\rtimes_r G) = \OO_X \rtimes_r G$ onto $\OO_{X\rtimes_r G}$ satisfying $$\Psi_r \circ (k_X \rtimes_r G) = k_{X \rtimes_r G}\, , \,  \Psi_r \circ (k_A\rtimes_r G) = k_{A \rtimes_r G}\,.$$ 
This homomorphism also satisfies  $\Psi_r \circ \big((k_X \rtimes_r G)  \times_\TT  (k_A \rtimes_r G)\big) = T_{X\rtimes_r G}\,.$ Using that $(\Phi_r)^{-1} = (t_X \rtimes_r G) \, \times_\TT \, (t_A \rtimes_r G)$, one checks without difficulty that $\Psi_r$ makes the diagram \eqref{Psi r} commute. 
Once we know $\Psi_r$ exists and makes the diagram \eqref{Psi r} commute,
it is unique since $T_{X} \rtimes_r G$ is surjective.

(2) Assume now that $(k_X\rtimes_r G,\,k_A\rtimes_r G)$ is Cuntz-Pimsner covariant. The homomorphism
$\Upsilon_r = (k_X\rtimes_r G) \times (k_A\rtimes_r G) : \OO_{X\rtimes_r G} \to \OO_X \rtimes_r G$ is then well defined and satisfies $\Upsilon_r \circ k_{X\rtimes_r G} = k_X \rtimes_r G, \,  \Upsilon_r \circ k_{A\rtimes_r G} = k_A \rtimes_r G$. It makes the diagram \eqref{surjection 2 red} commute because \eqref{triangle red} commutes. One checks immediately on generators that 
$\Psi_r$ and $\Upsilon_r$ are inverses of each other, so $\Psi_r$  an isomorphism, as asserted.
Once we know $\Upsilon_r$ exists and makes the diagram \eqref{surjection 2 red} commute,
it is unique since $T_{X\rtimes_r G}$ is surjective.
\end{proof}

When  $G$ is amenable, as it is in \cite{HN}, all the involved full crossed products coincide with their respective reduced crossed products, and Hao and Ng prove  that $(k_X\rtimes G,\,k_A\rtimes G)=(k_X\rtimes_r G,\,k_A\rtimes_r G)$ is Cuntz-Pimsner covariant in this case. An important step in their argument is to show  that the equality $J_X \rtimes G= J_{X\rtimes G}$ holds, hence that the equality $J_X \rtimes_r G= J_{X\rtimes_r G}$ also holds, when $G$ is amenable. 

\smallskip It seems worth recording the following related general result.

\begin{thm} \label{reduced2} 
We always have $J_X\rtimes_r G\subset J_{X\rtimes_r G}$. 

\smallskip  If $J_X\rtimes_r G\supset J_{X\rtimes_r G}$ $($or, equivalently, if $J_X\rtimes_r G  = J_{X\rtimes_r G})$,
then $(k_X\rtimes_r G,\,k_A\rtimes_r G)$ is Cuntz-Pimsner covariant, and the homomorphism $\Psi_r$ from \thmref{reduced} is an isomorphism from  $\OO_X \rtimes_r G$ onto $\OO_{X\rtimes_r G}$.  
\end{thm}

\begin{proof}
To prove that $J_X\rtimes_r G\subset J_{X\rtimes_r G}$, it suffices to repeat the argument given by Hao and Ng in the beginning of their proof of \cite[Proposition 2.7]{HN}. (They tacitly switch to the reduced case in this argument, as they may, since they assume that $G$ is amenable). 
 
 \smallskip Assume that  $J_X\rtimes_r G\supset J_{X\rtimes_r G}$. To check that the Toeplitz representation $(k_X\rtimes_r G,k_A\rtimes_r G)$
is Cuntz-Pimsner covariant, we have to show that
\[
(k_X\rtimes_r G)^{(1)}\circ \phi_{A\rtimes_r G}=k_A\rtimes_r G
\]
on the Katsura ideal $J_{X\rtimes_r G}$.
Employing the isomorphism $\KK(X\rtimes_r G)\cong \KK(X)\rtimes_r G$,
since the diagram
\[
\xymatrix@C+20pt{
A\rtimes_r G \ar[r]^-{\phi_{A\rtimes_r G}} \ar[dr]_{\phi_A\rtimes_r G}
&M_{A\rtimes_r G}(\KK(X\rtimes_r G)) \ar[r]^-{(k_X\rtimes_r G)^{(1)}} \ar[d]_\simeq
&M(\OO_X\rtimes_r G)
\\
&M_{A\rtimes_r G}(\KK(X)\rtimes_r G) \ar[ur]_{k_X^{(1)}\rtimes_r G}
}
\]
commutes,
it suffices to show that
\begin{equation}\label{equal red}
(k_X^{(1)}\rtimes_r G)\circ (\phi_A\rtimes_r G)=k_A\rtimes_r G
\end{equation}
on $J_{X\rtimes_r G}$\,.
Now, since the equality $k_X^{(1)}\circ\, \phi_A=k_A$ on $J_X$ is $G$-equivariant,
taking reduced crossed products gives
\begin{equation}\label{equal2 red}
(k_X^{(1)}\circ \phi_A)\rtimes_r G=k_A\rtimes_r G\midtext{on}J_X\rtimes_r G.
\end{equation}
We have $J_{X\rtimes_r G}\subset J_X\rtimes_r G$ by hypothesis,
and
\[
(k_X^{(1)}\rtimes_r G)\circ (\phi_A\rtimes_r G)=(k_X^{(1)}\circ \phi_A)\rtimes_r G,
\]
so \eqref{equal red} follows from \eqref{equal2 red}.
\end{proof}

\begin{cor}\label{cp regular}
Suppose that 
$A\rtimes G=A\rtimes_r G$
and $J_X\rtimes_r G\supset J_{X\rtimes_r G}$.
Then 
\[
\OO_X\rtimes G=\OO_X\rtimes_r G.
\]
\end{cor}

\begin{proof}
By \thmref{toep regular},
$\Lambda_X:X\rtimes G\to X\rtimes_r G$ is an isomorphism.
Since $J_X$ is a $G$-invariant ideal of $A$
and $\Lambda_A$ is faithful,
the regular representation $\Lambda_{J_X}$ is an isomorphism.
Moreover, by  Proposition \ref{reduced} we have $J_X\rtimes_r G=J_{X\rtimes_r G}$,
and hence $J_X\rtimes G=J_{X\rtimes G}$\,.
Thus, by Theorems~\ref{full}, \ref{reduced} and  \ref{reduced2},
we have a commutative diagram
\[
\xymatrix{
\OO_X\rtimes G \ar[d]_{\Lambda_{\OO_X}}
&\OO_{X\rtimes G} 
\ar[d]^{(k_{X\rtimes_r G}\circ \Lambda_X)\times (k_{A\rtimes_r G}\circ \Lambda_A)}_\simeq
\ar[l]_-\Upsilon^-\simeq
\\
\OO_X\rtimes_r G
&\OO_{X\rtimes_r G}, \ar[l]^-{\Upsilon_r}_-\simeq
}
\]
and it follows that $\Lambda_{\OO_X}$ is an isomorphism.
\end{proof}

%
%


%
%

It is not clear whether the equality $J_X \rtimes_r G = J_{X\rtimes_r G}$ holds in general. Anyhow, here is a result in this direction, probably close to what Katsura might have had in mind  in his comment to Hao and Ng about exact groups that we mentioned in the Introduction.

\begin{thm}\label{exact}
%

Suppose that $G$ is discrete. 

\begin{enumerate}
\item\label{thm:discreteexact} If $G$ is exact \cite{BO}, then $J_X\rtimes_r G = J_{X\rtimes_r G}$ and 
\[
\OO_X\rtimes_r G\, \cong\, \OO_{X\rtimes_r G}.
\]
\item\label{thm:discreteEAP} If $G$ has Exel's Approximation Property \cite{exelamenable, exelng}, then $J_X\rtimes_r G = J_{X\rtimes_r G}$ and
\[
 \OO_X\rtimes G=\OO_X\rtimes_r G\, \cong \, \OO_{X\rtimes_r G}=\OO_{X\rtimes G}.
\]
\end{enumerate}

\end{thm}

When $G$ is discrete, the canonical map $i_A^r$ embeds $A$ in the reduced crossed product $A\rtimes_r G$,
and we will identify $A$ with its image in $A\rtimes_r G$.
Also, we'll write $u$ for the canonical unitary homomorphism $i_G^r:G\to M(A\rtimes_r G)$,
so that
\[
A\rtimes_r G=\clspn\{au_s:a\in A,s\in G\}.
\]
There is a unique faithful conditional expectation $E:A\rtimes_r G\to A$ such that
$E(f)=f(e)$ for $f\in C_c(G,A)$,
and which is also characterized by
\[
E(au_s)=\begin{cases}a\case s=e\\0\case s\ne e.\end{cases}
\]

Now we recall a few essential facts from \cite{exelamenable} and \cite{exelexactarxiv} where
Exel studies Fell bundles over discrete groups. By considering the semidirect product Fell bundle $A\times G$ naturally associated to the action of $G$ on $A$, we may write $A\rtimes_r G$ as the reduced cross sectional algebra of this Fell bundle (cf.\ \cite[Proposition 3.8]{exelamenable}) and apply Exel's results in our situation. 
 
Let  $\mathcal{J}$ be an ideal of $A\rtimes_r G$. The ideal $\mathcal{J}$ is called \emph{induced} if $\mathcal{J} = J \rtimes_r G$ for some $G$-invariant ideal $J$ of $A$. It is called \emph{invariant} if $E(\mathcal{J}) \subset \mathcal{J}$, or, equivalently, if $E(\mathcal{J}) = \mathcal{J} \cap A$.  It is clear that any induced ideal is invariant, but it is unknown whether the converse holds in general. 
However, it follows easily from \cite[Theorem~5.1]{exelexactarxiv}
and
\cite[Proposition~4.10]{exelamenable}, respectively, that if
$G$ is exact or
the action of $G$ on $A$ has Exel's Approximation Property, 
then  $\mathcal{J}$ is induced whenever it is invariant, in which case we have $\mathcal{J} = E(\mathcal{J}) \rtimes_r G$.

\begin{proof}[Proof of \thmref{exact}]
Suppose $G$ is exact or the action of $G$ has Exel's approximation property. We will first show that $E(J_{X\rtimes_r G}) \subset J_X$. 
Since $J_X \subset J_X\rtimes_r G \subset J_{X\rtimes_r G}$, this will show that $J_{X\rtimes_r G}$ is invariant. Appealing to the results of Exel recalled above, we will get that $ J_{X\rtimes_r G}$ is an induced ideal. Moreover, as we also have $J_X = E(J_X) \subset E(J_{X\rtimes_r G})$, this will give us that $E(J_{X\rtimes_r G}) = J_X$. Hence, we will be able to conclude that
$$J_{X\rtimes_r G} = E(J_{X\rtimes_r G}) \rtimes_r G = J_X \rtimes_r G \, ,$$
so $J_X\rtimes_r G = J_{X\rtimes_r G}$, as asserted in both (a) and (b).


To show that $E(J_{X\rtimes_r G})\subset J_X$, we take $b\in J_{X\rtimes_r G}$.
We need to show:
\begin{enumerate}
\renewcommand{\labelenumi}{(\roman{enumi})}
\item $E(b)\in (\ker \phi_A)\ann$;
\item $\phi_A(E(b))\in \KK(X)$.
\end{enumerate}
For (i), let $a\in \ker \phi_A$.
We have $b\in (\ker \phi_{A\rtimes_r G})\ann$
and
\[
\ker \phi_A\subset \ker \phi_{A\rtimes_r G}
\]
because $G$ is discrete,
so $ba=0$, and hence
\[
E(b)a=E(ba)=0.
\]
Thus $E(b)\in (\ker \phi_A)\ann$.

For (ii),
let $E':\KK(X)\rtimes_r G\to \KK(X)$ be the canonical conditional expectation.
We have $\phi_{A\rtimes_r G}(b)\in \KK(X\rtimes_r G)$, so, modulo the isomorphism $\KK(X\rtimes_r G)\cong \KK(X)\rtimes_r G$,
\[
\phi_A(E(b))=E'(\phi_{A\rtimes_r G}(b))\in E'(\KK(X)\rtimes_r G)=\KK(X)\,.
\]

Thus we have shown $J_X\rtimes_r G = J_{X\rtimes_r G}$. We may therefore apply \thmref{reduced2} and get
$\OO_X\rtimes_r G\, \cong\, \OO_{X\rtimes_r G}\,.$ 
Finally, if we know that the action of $G$ on $A$ has Exel's Approximation Property, then, by \cite[Theorem~4.6]{exelamenable} (or \cite[Theorem 3.9]{exelng}) we have
$A\rtimes G=A\rtimes_r G$; hence, in this case, using also  \thmref{toep regular} and Corollary \ref{cp regular}, we get
\[
\OO_X\rtimes G=\OO_X\rtimes_r G\, \cong \, \OO_{X\rtimes_r G}=\OO_{X\rtimes G}\,.
\qedhere
\]
\end{proof}


As in the Toeplitz case, we include two results concerning nuclearity and exactness.

  \begin{prop}
 Assume that the action of $G$ on $A$ has Exel's Approximation Property and $A$ is nuclear. 
  Then 
$\OO_{X\rtimes G} =\OO_{X \rtimes_r G}$  is nuclear. 
 Moreover, if we assume in addition that $G$ is  discrete, then $\OO_X\rtimes G=\OO_X\rtimes_r G$ is also nuclear.
\end{prop}

\begin{proof}
It follows from \cite[Theorem 4.4]{exelng} that $A\rtimes G= A\rtimes_r G$ is nuclear. Hence, combining  \cite[Corollary 7.4]{KatsuraCorrespondence}  with \thmref{cp regular} gives the first assertion.  \thmref{exact} then gives the second assertion. 
\end{proof}

\begin{prop} 
Consider the following conditions:
 \begin{enumerate}
 \item $G$ is exact and $A$ is exact;
 \item $\OO_{X} \rtimes_r G$ is exact;
 \item$\OO_{X \rtimes_r G}$ is exact.
 \end{enumerate}
 Then we have $(1) \Rightarrow (2) \Rightarrow (3)$. 
 If $G$ is discrete, then all three conditions are equivalent.
 \end{prop}
 \begin{proof} The proof is quite similar to the proof of Proposition \ref{toep exact}.
 Assume that (1) holds. Then, by \cite[Theorem 7.1]{KatsuraCorrespondence}, $\OO_X$ is exact,  so \cite[Theorem 7.2]{anan:exact} gives that (2) holds. Since exactness passes to quotients \cite{BO}, it follows from  Theorem \ref{reduced} that (2) $\Rightarrow$ (3). If $G$ is discrete and (3) holds, then \cite[Theorem 7.1]{KatsuraCorrespondence} gives that $A\rtimes_r G$ is exact, so (1) holds. \end{proof}

\section{Concluding remarks}

We conclude with a discussion of the problem that originally motivated this work: Is  $\OO_X\rtimes_\beta G\cong \OO_{X\rtimes_\gamma G}$ in general? Hao and Ng have shown \cite[Theorem~2.10]{HN} that the answer is yes if $G$ is amenable.
In Theorem \ref{full} we expand on this to show that we have the desired isomorphism whenever $J_{X\rtimes G} = J_X \rtimes G$. We do not know  whether this is true in general.

Problems arise even when just considering whether we have an inclusion $J_X \rtimes G \subset J_{X\rtimes G}$. By definition, the Katsura ideal $J_{X\rtimes G}$ is the largest ideal of $A\rtimes G$ that is mapped by the left action $\varphi_{A \rtimes G}$ injectively into $\KK(X\rtimes G)$. Given an action $\gamma$ of $G$ on $X$, there is an induced action, usually denoted $\gamma^{(1)}$, of $G$ on $\KK(X)$ and there is always an isomorphism $\KK(X\rtimes G) \cong \KK(X) \rtimes G$ (see \cite{Combes1984} for example). Moreover, the diagram
\[
\xymatrix@C+30pt{
A\rtimes G \ar[r]^-{\varphi_{A\rtimes G}} \ar[dr]_{\varphi_A\rtimes G}
&M\bigl(\KK(X\rtimes G)\bigr) \ar[d]^{\tau}_\cong
\\
&M\bigl(\KK(X)\rtimes G\bigr)
}
\]
commutes.
Thus, the question becomes whether $\varphi_A\rtimes G$ maps $J_X\rtimes G$ injectively into $\KK(X)\rtimes G$.
We know that $\varphi_A$ maps $J_X$ injectively into $\KK(X)$,
so certainly $\varphi_A\rtimes G$ maps $J_X\rtimes_\alpha G$ \emph{into} $\KK(X)\rtimes G$ ---
the remaining issue is whether 
$\varphi_A\rtimes G$ is injective on $J_X\rtimes_\alpha G$.

What could possibly go wrong?
Well, $J_X$ is an ideal of $A$, but its (faithful) image in $\KK(X)$ is only a $G$-invariant $C^*$-subalgebra.
Now we run up against the following unpleasant behavior:
if $C^*(G)$ is nonnuclear then there exist an action of $G$ on a $C^*$-algebra $B$ and a $G$-invariant $C^*$-subalgebra $C\subset B$ such that $C\rtimes G\not\subset B\rtimes G$.
In fact, we can do it with a trivial action, because (see, e.g., \cite[Theorem~IV.3.1.12]{blackadar}) we can have
\[
C\otimes_{\max} C^*(G)\not\subset B\otimes_{\max} C^*(G).
\]
Our situation is special so this does not necessarily provide a counter-example, but it suggests that it is a difficult question in general.

\providecommand{\bysame}{\leavevmode\hbox to3em{\hrulefill}\thinspace}
\providecommand{\MR}{\relax\ifhmode\unskip\space\fi MR }
\providecommand{\MRhref}[2]{%
  \href{http://www.ams.org/mathscinet-getitem?mr=#1}{#2}
}
\providecommand{\href}[2]{#2}


\begin{thebibliography}{EKQR06}

\bibitem[AD87]{ana}
C.~Anantharaman-Delaroche, \emph{{Syst{\`e}mes dynamiques non commutatifs et
  moyennabilit{\'e}}}, Math. Ann. \textbf{279} (1987), 297--315.

\bibitem[AD02]{anan:exact}
\bysame,
\emph{Amenability and exactness for dynamical
  systems and their {$C^\ast$}-algebras}, Trans. Amer. Math. Soc. \textbf{354}
  (2002), no.~10, 4153--4178 .

\bibitem[BC12]{bc:regular}
E.~B{\'e}dos and R.~Conti, \emph{On discrete twisted {$\rm
  C^*$}-dynamical systems, {H}ilbert {$\rm C^*$}-modules and regularity},
  M\"unster J. Math. \textbf{5} (2012), 183--208. 

\bibitem[Bla06]{blackadar}
B.~Blackadar, \emph{Operator algebras}, Encyclopaedia of Mathematical Sciences,
  vol. 122, Springer-Verlag, Berlin, 2006, Theory of $C{^{*}}$-algebras and von
  Neumann algebras, Operator Algebras and Non-commutative Geometry, III.

\bibitem[BO08]{BO}
N.~P. Brown and N.~Ozawa, \emph{{$C^*$}-algebras and finite-dimensional
  approximations}, Graduate Studies in Mathematics, vol.~88, American
  Mathematical Society, Providence, RI, 2008.

\bibitem[Com84]{Combes1984}
F.~Combes, \emph{Crossed products and {M}orita equivalence}, Proc. London Math.
  Soc. (3) \textbf{49} (1984), no.~2, 289--306. 

\bibitem[DKQ12]{dkq}
V.~Deaconu, A.~Kumjian, and J.~Quigg, \emph{Group actions on topological
  graphs}, Ergodic Theory Dynam. Systems \textbf{32} (2012), 1527--1566.

\bibitem[EKQR00]{taco}
S.~Echterhoff, S.~Kaliszewski, J.~Quigg, and I.~Raeburn, \emph{{Naturality and
  induced representations}}, Bull. Austral. Math. Soc. \textbf{61} (2000),
  415--438.

\bibitem[EKQR06]{enchilada}
\bysame, \emph{{A Categorical Approach to Imprimitivity Theorems for
  C*-Dynamical Systems}}, vol. 180, Mem. Amer. Math. Soc., no. 850, American
  Mathematical Society, Providence, RI, 2006.

\bibitem[Exe97]{exelamenable}
R.~Exel, \emph{{Amenability for Fell bundles}}, J. reine angew. Math.
  \textbf{492} (1997), 41--73.

\bibitem[Exe]{exelexactarxiv}
\bysame,
\emph{{Exact groups, induced ideals, and Fell bundles}},
  arXiv:math/0012091 [math.OA].

\bibitem[EN02]{exelng}
R.~Exel and C.-K. Ng, \emph{Approximation property of {$C\sp *$}-algebraic
  bundles}, Math. Proc. Cambridge Philos. Soc. \textbf{132} (2002), no.~3,
  509--522.

\bibitem[HN08]{HN}
G.~Hao and C.-K. Ng, \emph{Crossed products of {$C^*$}-correspondences by
  amenable group actions}, J. Math. Anal. Appl. \textbf{345} (2008), no.~2,
  702--707.

\bibitem[KQR]{kqrcoact}
S.~Kaliszewski, J.~Quigg, and D.~Robertson, \emph{{Coactions on Cuntz-Pimsner
  algebras}}, Math. Scand., accepted.

\bibitem[KQR13]{kqrfunctor}
\bysame, \emph{{Functoriality of Cuntz-Pimsner correspondence maps}}, J. Math.
  Anal. Appl. \textbf{405} (2013), 1--11.

\bibitem[Kas88]{K}
G.~G. Kasparov, \emph{Equivariant {$KK$}-theory and the {N}ovikov conjecture},
  Invent. Math. \textbf{91} (1988), no.~1, 147--201.

\bibitem[Kat04]{KatsuraCorrespondence}
T.~Katsura, \emph{On {$C^*$}-algebras associated with {$C^*$}-correspondences},
  J. Funct. Anal. \textbf{217} (2004), no.~2, 366--401.

\bibitem[Kat07]{KatsuraIdeal}
\bysame,
\emph{Ideal structure of {$C^*$}-algebras associated with
  {$C^*$}-correspondences}, Pacific J. Math. \textbf{230} (2007), no.~1,
  107--145. 

\bibitem[Kim14]{kim:coaction}
D.~W. Kim, \emph{Coactions of hopf {$C^\ast$}-algebras on {C}untz-{P}imsner
  algebras}, 
  arXiv:1407.6106 [math.OA].

\bibitem[Lan95]{lan:hilbert}
E.~C. Lance, \emph{{Hilbert $C^*$-modules}}, London Math. Soc. Lecture Note
  Ser., vol. 210, Cambridge University Press, 1995.

\bibitem[Lan79]{lan:dual}
M.~B. Landstad, \emph{{Duality theory for covariant systems}}, Trans. Amer.
  Math. Soc. \textbf{248} (1979), 223--267.

\bibitem[QS92]{qs:regularity}
J.~C. Quigg and J.~Spielberg, \emph{{Regularity and hyporegularity in
  $C^*$-dynamical systems}}, Houston J. Math. \textbf{18} (1992), 139--152.

\bibitem[RW98]{tfb}
I.~Raeburn and D.~P. Williams, \emph{{Morita equivalence and continuous-trace
  $C^*$-algebras}}, Math. Surveys and Monographs, vol.~60, American
  Mathematical Society, Providence, RI, 1998.

\bibitem[Wil07]{danacrossed}
D.~P. Williams, \emph{Crossed products of {$C{\sp \ast}$}-algebras},
  Mathematical Surveys and Monographs, vol. 134, American Mathematical Society,
  Providence, RI, 2007.

\end{thebibliography}
\end{document}